\documentclass[11pt]{amsart}
\usepackage{latexsym,amscd,amssymb, graphicx, color, amsthm, bm, amsmath, cancel, enumitem}  
\usepackage{tikz}
\usepackage[margin=1in]{geometry}
\usepackage{hyperref}
\usepackage[all,cmtip]{xy}

\numberwithin{equation}{section}

\newtheorem{theorem}{Theorem}[section]
\newtheorem{proposition}[theorem]{Proposition}
\newtheorem{corollary}[theorem]{Corollary}
\newtheorem{lemma}[theorem]{Lemma}

\newtheorem{problem}[theorem]{Problem}

\newtheorem{remark}[theorem]{Remark}

\theoremstyle{definition}
\newtheorem{defn}[theorem]{Definition}

\newcommand{\symm}{{\mathfrak{S}}}
\newcommand{\OOO}{{\mathcal{O}}}
\newcommand{\EEE}{{\mathcal{E}}}
\newcommand{\KKK}{{\mathcal{K}}}
\newcommand{\CC}{{\mathbb{C}}}
\newcommand{\ZZ}{{\mathbb{Z}}}
\newcommand{\PP}{{\mathbb{P}}}

\newcommand{\sell}{{\mathfrak{sl}}}

\newcommand{\vol}{{\mathrm{vol}}}

\begin{document}

\title[The combinatorics of supertorus sheaf cohomology]
{The combinatorics of supertorus sheaf cohomology}

\author{Jesse Kim}
\author{Jeffrey M. Rabin}
\author{Brendon Rhoades}
\address
{Department of Mathematics \newline \indent
University of California, San Diego \newline \indent
La Jolla, CA, 92093, USA}
\email{(jvkim, jrabin, bprhoades)@ucsd.edu}

\begin{abstract}
Affine superspace $\CC^{1 \mid n}$ has a single bosonic coordinate $z$ and $n$ fermionic coordinates $\theta_1, \dots, \theta_n$.
Let $M$ be the supertorus obtained by quotienting  $\CC^{1 \mid n}$ by the abelian group generated by the  maps 
$S: (z,\theta_1, \dots, \theta_n) \mapsto (z + 1, \theta_1, \dots, \theta_n)$ and
$T: (z, \theta_1, \dots, \theta_n) \mapsto (z + t, \theta_1 + \alpha_1, \dots, \theta_n + \alpha_n)$ where
$t \in \CC$ has positive imaginary part and $\alpha_1, \dots, \alpha_n$ are independent fermionic parameters.
We 
compute the zeroth and first cohomology groups of the structure sheaf $\OOO$ of $M$ as 
doubly graded $\symm_n$-modules,
exhibiting an instance of Serre duality between these groups.
We use skein relations and noncrossing matchings to give a combinatorial
 presentation of $H^0(M,\OOO)$ in terms of generators and relations.
\end{abstract}

\maketitle

\section{Introduction}
\label{Introduction}

In classical algebraic geometry, regular functions on complex varieties $X$ are represented by 
rational functions $f$  in commuting variables $x_1, x_2, \dots$ whose denominators have no zeros on $X$.
The prototypical model is affine $m$-space $\CC^m$ with coordinate ring $\CC[x_1, \dots, x_m]$.
In supergeometry, one considers functions of both commuting variables $x_1, x_2, \dots $ representing bosonic coordinates
as well as anticommuting variables $\theta_1, \theta_2, \dots$ representing fermionic coordinates. The coordinate ring of affine superspace $\CC^{m \mid n}$ 
is the tensor product $\CC[x_1, \dots, x_m] \otimes \wedge \{ \theta_1, \dots, \theta_n \}$ of a polynomial ring of rank $m$ and an exterior algebra of rank $n$.

Let $\CC^{1 \mid n}$ be affine superspace with one bosonic coordinate $z$ and $n$ fermionic coordinates $\theta_1, \dots, \theta_n$.
We consider a free exterior algebra $\wedge \{ \alpha_1, \dots, \alpha_n \}$ 
generated by ``coefficient" fermionic variables $\alpha_1, \dots, \alpha_n$.  In supergeometry, elements of $\wedge \{ \alpha_1, \dots, \alpha_n \}$ of odd homogeneous degree
may be substituted for the $\theta_i$ in any regular functions $f(z,\theta_1, \dots, \theta_n)$ on $\CC^{1 \mid n}$; this is the anticommutative analogue of evaluating 
a polynomial $\CC[x_1, \dots, x_m ]$ at a point in $\CC^m$.

Let $t \in \CC$ be a complex number with positive imaginary part.  We consider two translation operators $S,T: \CC^{1 \mid n} \rightarrow \CC^{1 \mid n}$ defined in terms
of coordinates by
\begin{align}
\label{Stranslation} 
S:& (z, \theta_1, \dots, \theta_n) \mapsto (z+1, \theta_1, \dots, \theta_n), \\
\label{Ttranslation} 
T:& (z, \theta_1, \dots, \theta_n) \mapsto (z + t, \theta_1 + \alpha_1, \dots, \theta_n + \alpha_n).
\end{align}
We let $M$ be the quotient of $\CC^{1 \mid n}$ by the group of operators generated by $S$ and $T$.
If $n = 0$ and no fermionic variables were present, the space $M$ would be a copy of the usual torus $S^1 \times S^1$
with complex structure given by the modulus $t$. 
As such, we may regard $M$ as a superspace extension of the torus (technically, a family of supertori over the parameter space $\wedge \{ \alpha_1, \dots, \alpha_n \}$).

Let $\OOO$ be the structure sheaf of $M$, so that the zeroth cohomology $H^0(M,\OOO)$  is the algebra of globally defined regular functions on $M$ 
(see Definition~\ref{h0-definition} for a precise definition of $H^0(M,\OOO)$).
If $n = 0$, the space $H^0(M,\OOO)$ is merely a copy of the ground field $\CC$; there are no nonconstant regular functions on the torus.
On the other hand,  we show that $H^0(M,\OOO)$ has rich combinatorial structure in the supertorus case of $n > 0$.
\begin{itemize}
\item  We show (Corollary~\ref{cor:basis}) that $H^0(M,\OOO)$ is generated as a $\CC$-algebra by the `basic invariants'
$\alpha_i, \alpha_i \theta_i,$ and $\alpha_i \theta_j + \alpha_j \theta_i$ for $1 \leq i < j \leq n$.
Furthermore, we give a combinatorial description
(Theorem~\ref{h0-presentation}) of the relations which hold among these invariants in terms of the classical Ptolemy relation 
\cite{Kauffman, KR2}
which resolves a simple crossing as a sum 
of its two possible resolutions.
\item Considering degree in the coefficient variables $\alpha_i$ as well as the 
$\theta_i$, the ring $H^0(M,\OOO)$ attains the structure of a bigraded module over the symmetric group $\symm_n$.
We compute its bigraded isomorphism type (Theorem~\ref{h0-module-structure}).  In particular, the $(i,j)$-component
$H^0(M,\OOO)_{i,j}$ is nonzero if and only if $i \geq j$, in which case it has dimension ${n \choose i}{n \choose j} - {n \choose i+1} {n \choose j-1}$.
Therefore, the dimensions of the spaces $H^0(M,\OOO)_{i,i}$ are the Narayana numbers $\mathrm{Nar}(n+1,i+1)$ for $i = 1, \dots, n$ and the sum
$\sum_{i = 0}^n H^0(M,\OOO)_{i,i}$ of these dimensions is the $(n+1)^{st}$ Catalan number $\mathrm{Cat}(n+1)$.
\end{itemize}

It follows from the second bullet point above that the vector spaces $H^0(M,\OOO)_{i,j}$ and $H^0(M,\OOO)_{n-j,n-i}$ have the same dimension
for all $i+j \leq n$.
We prove (Theorem~\ref{lefschetz-theorem}) that multiplication by $\ell^{n-i-j}$ where $\ell = \alpha_1 \theta_1 + \cdots + \alpha_n \theta_n$
furnishes a linear isomorphism between these spaces.
This fact may be viewed as a bigraded version of the Hard Lefschetz Theorem for the ring $H^0(M,\OOO)$.
We also show that the first cohomology group $H^1(M,\OOO)$ (see Definition~\ref{h1-definition}) and $H^0(M,\OOO)$ satisfy 
Serre duality (Theorem~\ref{perfect-pairing}).
The $n = 1$ case of these results was proven by the second author \cite{Rabin} in previous work.

We provide a bit more background in supergeometry, although this is not required for the rest of this paper.
The $n=1$ supertorus considered in \cite{Rabin} was in fact a super Riemann surface, with the second group generator above modified to $T: (z, \theta_1) \mapsto (z + t + \theta_1 \alpha_1, \theta_1 + \alpha_1)$. 
In principle, functions in $H^0(M,\OOO)$ or cocycles in $H^1(M,\OOO)$ might depend on $z$ as well as $\alpha_1, \theta_1$.
However, one can show by a Fourier series argument \cite{BR, Hodgkin, Rabin} (using the periodicity given by $S$) that they are in fact independent of $z$. 
This implies in particular that the extra even term $ \theta_1 \alpha_1$ in $T$ does not affect the cohomology and can be ignored. 
It also makes the generator $S$ trivial and allows us to compute the cohomology from $T$ alone acting on the fermionic variables. 
In this paper we will compute $H^1(M,\OOO)$ as the group cohomology of the cyclic group generated by $T$.
The equivalence of sheaf cohomology and group cohomology in this instance follows from \cite{Mumford, Silverman}.

Serre duality for supermanifolds has been proven in \cite{HaskeWells} for individual supermanifolds and in \cite{BR} for families. 
The duality we exhibit between $H^1(M,\OOO)$ and $H^0(M,\OOO)$ is an instance of this general structure
(technically, as applied to the total space of our family $M$ of supertori). 

Our results can be viewed as a first step toward a general theory of Abelian supervarieties, including the Jacobians of supercurves. 
In such a context the two generators $S,T$ would be extended to a set of $2g$ supertranslations 
$S_1, \ldots, S_g, T_1, \ldots, T_g$ representing a canonical homology basis of a torus of complex dimension $g$. 
The cohomology of the group they generate would be correspondingly more complicated, although still independent of the $g$ even coordinates and hence a purely algebraic problem. 
The computation of $H^1(M,\OOO)$ is of particular interest in view of its role in classifying topologically trivial line bundles 
on $M$ \cite{BR}.

\section{Background}
\label{Background}

We will need a notion of fermionic differentiation. 
Let $(\omega_1, \dots, \omega_m)$
be a list of $m$ fermionic variables and let $\wedge \{ \omega_1, \dots, \omega_m \}$ be the exterior algebra over these variables.  
For $1 \leq i \leq m$,
we define an operator $\partial/\partial \omega_i$ on $\wedge \{ \omega_1, \dots, \omega_m \}$ as the linear extension of
\begin{equation}
\partial/\partial \omega_i (\omega_{j_1} \cdots \omega_{j_r}) := \begin{cases}
(-1)^{s-1} \omega_{j_1} \cdots \widehat{\omega_{j_s}} \cdots \omega_{j_r} & \text{if $j_s = i$} \\
0 & \text{if $i \neq j_1, \dots, j_s$}
\end{cases}
\end{equation}
whenever $1 \leq j_1, \dots, j_s \leq m$ are distinct indices. 
Here the hat denotes omission. 

Let $G$ be a finite group. The {\em Grothendieck ring} of $G$ is the free $\ZZ$-module with basis given by isomorphism classes $[V]$
of irreducible $\CC[G]$-modules $V$. We extend the notation $[-]$ to arbitrary finite-dimensional $\CC[G]$-modules 
by setting $[V] = [W] + [U]$ whenever we have a short exact sequence $0 \rightarrow U \rightarrow V \rightarrow W \rightarrow 0$
of finite-dimensional $\CC[G]$-modules. The product structure on the Grothendieck ring is given by
\begin{equation}
[V] \cdot [W] := [V \otimes W]
\end{equation}
where $G$ acts on the vector space $V \otimes W$ by $g \cdot (v \otimes w) := (g \cdot v) \otimes (g \cdot w)$ for 
$g \in G, v \in V,$ and $w \in W$.
We will consider the Grothendieck ring of $G = \symm_n$ in this paper; this is a free $\ZZ$-algebra with basis indexed by the set of 
partitions $\lambda \vdash n$.
The structure constants for the canonical basis $[V^\lambda]$ of irreducible $\symm_n$-modules are the {\em Kronecker coefficients};
these are famously difficult to compute.

 We shall need a result guaranteeing the invertibility of a certain combinatorial matrix. Recall that the {\em Boolean poset}
$B_n$ has elements given by subsets $S \subseteq \{1, \dots, n \}$ and order relation $S \leq T$ if and only if $S \subseteq T$.
Let $B_{n,i}$ denote the family of size $i$ subsets of $\{1, \dots, n \}$; these are the rank $i$ elements in $B_n$.
The next result states that the incidence matrix between the complementary ranks $B_{n,i}$ and $B_{n,n-i}$ is invertible.

\begin{theorem}
\label{boolean-invertible}
For $0 \leq i \leq j \leq n$, let $M_n(i,j)$ be the ${n \choose i} \times {n \choose j}$ matrix with rows indexed by $B_{n,i}$,
 columns indexed by $B_{n,j}$, and entries determined by the rule
 \begin{equation}
 M_n(i,j)_{S,T} = \begin{cases}
 1 & S \subseteq T \\
 0 & \text{otherwise.}
 \end{cases}
 \end{equation}
 For all $0 \leq i \leq n/2$, the matrix $M_n(i,n-i)$ is invertible.
\end{theorem}

The origins of Theorem~\ref{boolean-invertible} are difficult to trace.
In \cite{Stanley}, Stanley used his theory of differential posets to calculate the (nonzero) eigenvalues 
of $M_n(i,n-i)$.
Another proof of Theorem~\ref{boolean-invertible} is due to Hara and Watanabe \cite{HW}.
In \cite{HW} Theorem~\ref{boolean-invertible} is viewed as a witness for the Hard Lefschetz property of 
the cohomology ring of the $n$-fold self product $\PP^1 \times \cdots \times \PP^1$ of the Riemann sphere.

\section{The bigraded $\symm_n$-structure of $H^0(M,\OOO)$ and $H^1(M,\OOO)$}
\label{Module}

Let $(\alpha_1, \dots, \alpha_n)$ and $(\theta_1, \dots, \theta_n)$ be two lists of fermionic variables and consider
the rank $2n$ exterior algebra 
\begin{equation}
E_n := \wedge \{ \alpha_1, \dots, \alpha_n, \theta_1, \dots, \theta_n \}
\end{equation}
on these variables over the ground field $\CC$. This algebra is doubly graded, with
\begin{equation}
(E_n)_{i,j} = 
\wedge^i \{ \alpha_1, \dots, \alpha_n \} \otimes \wedge^j \{ \theta_1, \dots, \theta_n \}
\end{equation}
and carries an action of $\symm_n$ 
\begin{equation}
w \cdot \theta_i := \theta_{w(i)} \quad \quad
w \cdot \alpha_i := \alpha_{w(i)} \quad \quad w \in \symm_n, \, \, 1 \leq i \leq n
\end{equation}
which respects this bigrading.

The zeroth cohomology group
$H^0(M,\OOO)$ is the set of globally defined regular functions on $M$, namely the functions on $\CC^{1 \mid n}$ that are invariant under 
the translations $S$ and $T$ of Eqs. \eqref{Stranslation} and \eqref{Ttranslation}. 
Since these functions are known to be independent of the even coordinate $z$, it suffices to check invariance under 
the action of $T$ on the odd coordinates only. 

Thus, we consider the map of algebras
\begin{equation}
T:  E_n \rightarrow E_n
\end{equation}
defined on generators by
\begin{equation}
T: \theta_i \mapsto \theta_i + \alpha_i \quad \quad T: \alpha_i \mapsto \alpha_i.
\end{equation}
The map $T$ acts by fermionic translation.
The zeroth cohomology group $H^0(M,\OOO)$ is related to $T$ as follows.

\begin{defn}
\label{h0-definition}
Let $H^0(M,\OOO) := \{ f \in E_n \,:\, T(f) = f \}$ be the subalgebra of 
$E_n$ fixed by $T$.
\end{defn}

Despite the fact that $T$ is not bigraded, its fixed subspace
$H^0(M,\OOO)$ is a bigraded subalgebra of $E_n$.
To show this,
we introduce the operator 
$\tau: E_n \rightarrow E_n$ given by
\begin{equation}
\tau(f) := \sum_{i = 1}^n \alpha_i \cdot (\partial/\partial \theta_i) f
\end{equation}
The operator $\tau$ is bihomogeneous of degree $(1,-1)$.

\begin{proposition}
\label{h0-is-bigraded}
Let $f \in E_n$.  We have $T(f) = f$ if and only if $\tau(f) = 0$.
In particular, the subalgebra $H^0(M,\OOO)$ of $E_n$ is bigraded.
\end{proposition}

Our proof of Proposition~\ref{h0-is-bigraded} proceeds by showing that $\tau$ is an `infinitesimal' version of the translation operator $T$.

\begin{proof}
We claim that we have the equality of operators 
\begin{equation}
\label{exponential-expansion}
T = \mathrm{exp}(\tau) :=\mathrm{id} + \tau + \frac{1}{2!} \tau^2 + \frac{1}{3!} \tau^3 + \cdots 
\end{equation}
on  $E_n$.  Since $\tau^{n+1} = 0$ on $E_n$,
the RHS is really a finite sum. Indeed, this equality of linear operators may be easily checked on monomials using
the evaluation
\begin{equation}
T: \theta_{i_1} \cdots \theta_{i_r} \cdot \alpha_{j_1} \cdots \alpha_{j_s} \mapsto
(\theta_{i_1} + \alpha_{i_1}) \cdots (\theta_{i_r} + \alpha_{i_r}) \cdot \alpha_{j_1} \cdots \alpha_{j_s}.
\end{equation} 
If $\tau(f) = 0$, Equation~\eqref{exponential-expansion} immediately gives $T(f) = f$. 

On the other hand, if $T(f) = f$,
Equation~\eqref{exponential-expansion} gives the relation
\begin{equation}
\label{degree-is-off}
0 = \tau(f) + \frac{1}{2!} \tau^2(f) + \frac{1}{3!} \tau^3(f) + \cdots 
\end{equation}
inside $E_n$.
The $T$-invariant function $f$ may be written uniquely as a sum
$f = \sum_{i,j = 0}^n f_{i,j}$ where
$f_{i,j} \in (E_n)_{i,j}$.
If $f_{i,j} = 0$ for all $i, j$, then certainly $\tau(f) = 0$ and we are done. Otherwise, choose $j_0$ minimal such that
$f_{i,j_0} \neq 0$ for some $0 \leq i \leq n$ and let $f_{*,j_0} := \sum_{i = 0}^n f_{i,j_0}$.
Since $\tau$ has bidegree $(1,-1)$, Equation~\eqref{degree-is-off} and our choice of $j_0$ imply that 
$\tau(f_{*,j_0}) = 0$. Equation~\eqref{exponential-expansion} shows that $T(f_{*,j_0}) = f_{*,j_0}$, so that 
$T(f - f_{*,j_0}) = f - f_{*,j_0}$.
By induction on bidegree support, we have $\tau(f - f_{*,j_0}) = 0$.
We conclude that $\tau(f) = \tau(f_{*,j_0}) + \tau(f - f_{*,j_0}) = 0 + 0 = 0$.
\end{proof}

By Proposition~\ref{h0-is-bigraded}, we have a direct sum decomposition
\begin{equation}
H^0(M,\OOO) = \bigoplus_{i,j = 0}^n H^0(M,\OOO)_{i,j}
\end{equation}
in the category of bigraded rings or bigraded $\symm_n$-modules.
The module structure of $H^0(M,\OOO)_{i,j}$ is given as follows.

\begin{theorem}
\label{h0-module-structure}
We have $H^0(M,\OOO)_{i,j} = 0$ unless $i \geq j$.  If $i \geq j$, then 
\begin{align}
\left[ H^0(M,\OOO)_{i,j} \right] &= \left[   \wedge^i \CC^n \right] \cdot \left[   \wedge^j \CC^n  \right] -
 \left[   \wedge^{i+1} \CC^n  \right] \cdot \left[   \wedge^{j-1} \CC^n \right]  \\
 &= [ (E_n)_{i,j} ] - [ (E_n)_{i+1,j-1} ]
 \end{align}
 within the Grothendieck ring of $\symm_n$ where the action of $\symm_n$ on $\CC^n$ is by coordinate permutation. 
\end{theorem}

The bigraded $\symm_n$-structure of $H^0(M,\OOO)$ is very similar to that of the bigraded $\symm_n$-module $FDR_n$ of {\em fermionic diagonal coinvariants}
introduced in \cite{KR1} and further studied in \cite{Kim, KR2}.
The following proof should be compared to those of \cite[Thm. 3.2, Thm. 4.2]{KR1}.

\begin{proof}
Let $(i,j)$ be a fixed bidegree. For any $r \geq 0$, we have an operator
\begin{equation}
\tau^r: (E_n)_{i,j} \longrightarrow 
(E_n)_{i+r,j-r}
\end{equation}
When $i \leq j$ and $r = j-i$, we have a map
\begin{equation}
\label{complementary-composition}
\tau^{j-i}: (E_n)_{i,j} \longrightarrow
(E_n)_{j,i}
\end{equation}
between complementary bidegrees.

We claim that the map $\tau^{j-i}$ appearing in \eqref{complementary-composition} is bijective. To prove this, we examine its 
matrix with respect to a strategic choice of bases.  

Given subsets $A, B \subseteq \{1, \dots, n \}$, let
$m_{A,B} \in E_n$ be the monomial whose variables consist of $\alpha_a$ 
for $a \in A$ and $\theta_b$ for $b \in B$ written in increasing order with respect to
\begin{equation}
\alpha_1 < \theta_1 < \alpha_2 < \theta_2 < \cdots < \alpha_n < \theta_n.
\end{equation}
For example, if $n = 8, A = \{2,3,5\},$ and $B = \{1,3,4,6\}$ we have
\begin{equation*}
m_{A,B} = \theta_1 \cdot  \alpha_2 \cdot \alpha_3 \theta_3 \cdot \theta_4 \cdot \alpha_5 \cdot \theta_6. 
\end{equation*}
It is not hard to see that
\begin{equation}
\label{basic-transition}
\tau(m_{A,B}) = \sum_{\substack{c \notin A \\ c \in B}} m_{A  \, \cup \, c, B \, - \, c}.
\end{equation}
The definition of $m_{A,B}$ was chosen so that Equation~\eqref{basic-transition} is free of signs. Iterating Equation~\eqref{basic-transition}
$j-i$ times, we see that 
\begin{equation}
\label{iterated-transition}
\tau^{j-i}(m_{A,B}) = (j-i)! \cdot \sum_{\substack{|C| \, = \, j-i \\ C \,  \cap  \, A  \, = \, \varnothing \\ C \, \subseteq \, B}} m_{A  \, \cup \, C, B \, - \, C}.
\end{equation}
The matrix for $\tau^{j-i}$ with respect to the $m_{A,B}$ basis therefore breaks up into a direct sum of matrices
indexed by subsets $D = A \cup B$ where each direct summand is $(j-i)!$ times a matrix which Theorem~\ref{boolean-invertible}
guarantees is invertible.
We conclude that $\tau^{j-i}$ is invertible whenever $i \leq j$.

Whenever a composition $f \circ g$ of functions is a bijection, the map $f$ is a surjection and the map $g$ is an injection.
We conclude that 
\begin{equation}
\label{single-tau-step}
\tau: (E_n)_{i,j} \longrightarrow 
(E_n)_{i+1,j-1}
\end{equation}
is injective whenever $i < j$. Proposition~\ref{h0-is-bigraded} implies that $H^0(M,\OOO)_{i,j} = 0$ when $i < j$.
Similarly, if $i \geq j$, the map in Equation~\eqref{single-tau-step} is surjective.
Since $\tau$ is $\symm_n$-equivariant, Proposition~\ref{h0-is-bigraded} implies the desired relation in the Groethendieck ring 
of $\symm_n$.
\end{proof}

As a consequence of Theorem~\ref{h0-module-structure}, we have
\begin{equation}
\dim H^0(M,\OOO)_{i,j} =  \begin{cases}
{n \choose i} {n \choose j} - {n \choose i+1} {n \choose j-1} & i \geq j \\ 
0 & i < j
\end{cases}
\end{equation}
which implies that
\begin{equation}
\label{dimension-coincidence}
\dim H^0(M,\OOO)_{i,j} = \dim H^0(M,\OOO)_{n-j,n-i} \quad \quad \text{for $i+j \leq n$}.
\end{equation}
The following result shows that multiplication by the appropriate power of
\begin{equation}
\label{ell-definition}
\ell := \alpha_1 \theta_1 + \cdots + \alpha_n \theta_n
\end{equation}
yields a linear isomorphism $H^0(M,\OOO)_{i,j} \xrightarrow{\, \sim \, } H^0(M,\OOO)_{n-j,n-i}$.

\begin{theorem}
\label{lefschetz-theorem}
Suppose $i + j \leq n$.  We have a linear isomorphism
\begin{equation}
(-) \times \ell^{n-i-j}: H^0(M,\OOO)_{i,j} \xrightarrow{ \, \, \sim \, \, } H^0(M,\OOO)_{n-j,n-i}
\end{equation}
where $\ell \in E_n$ is given by \eqref{ell-definition}.
\end{theorem}

\begin{proof}
Since $\ell$ is $T$-invariant, we have $\ell \in H^0(M,\OOO)$ and the given map is well-defined. 
By the dimension equality \eqref{dimension-coincidence} it suffices to show that we have a linear isomorphism
\begin{equation}
\label{new-isomorphism}
(-) \times \ell^{n-i-j}: (E_n)_{i,j} \xrightarrow{ \, \, \sim \, \, } (E_n)_{n-j,n-i}.
\end{equation}
The proof that \eqref{new-isomorphism} is bijective is similar to that of Theorem~\ref{h0-module-structure}. For subsets $A, B \subseteq \{1,\dots,n\}$, let 
$m'_{A,B} \in E_n$ be the monomial given by
\begin{equation}
m'_{A,B} := \prod_{c \in A \cap B} \alpha_c \cdot \theta_c \times \prod_{a \in A - B} \alpha_a \times \prod_{b \in B-A} \theta_b
\end{equation}
where each product is taken in increasing order. For example, if $A = \{2,4,5,7\}$ and $B = \{3,4,7\}$ then
$m'_{A,B} = ( \alpha_4 \theta_4 \alpha_7 \theta_7) \cdot \alpha_2 \alpha_5 \cdot \theta_3$. It is not hard to see that
\begin{equation}
\label{basic-ell-transition}
\ell \cdot m'_{A,B} = \sum_{c \, \notin  \, A \cup B} m'_{A \cup c, B \cup c}
\end{equation}
where the definition of $m'_{A,B}$ guarantees that Equation~\eqref{basic-ell-transition} is free of signs.
The proof that the map in \eqref{new-isomorphism} is bijective now follows from the same reasoning as in the proof of 
Theorem~\ref{h0-module-structure}.
The isomorphism \eqref{new-isomorphism} restricts to give the isomorphism in the statement of the theorem.
\end{proof}

The first cohomology group $H^1(M,\OOO)$ of the sheaf $\OOO$ classifies line bundles over $M$ which are topologically trivial (i.e. whose Chern class vanishes).
In order to compute  $H^1(M,\OOO)$,
we recall some notions from group cohomology.

Let $G$ be an abelian group and let $N$ be a $G$-module. 
A {\em $1$-cocycle} is a function $c: G \rightarrow N$ which satisfies
\begin{equation}
c(gh) = h \cdot c(g) + c(h) \quad \quad \text{for all $g, h \in G$.}
\end{equation}
A $1$-cocycle $c: G \rightarrow N$ is determined by its values on a generating set of $G$.
A {\em $1$-coboundary} is a function $c_n: G \rightarrow N$ of the form
\begin{equation}
c_n(g) := g \cdot n - n
\end{equation}
for some $n \in N$. It is easily seen that every $1$-coboundary is a 1-cocycle. 
The sets of 1-cocycles and 1-coboundaries both form groups under pointwise addition.
The {\em first cohomology group} is the quotient
\begin{equation}
H^1(G,N) := \text{(1-cocycles)}/\text{(1-coboundaries)}
\end{equation}
of the group of 1-cocycles by its subgroup of 1-coboundaries.

From the equivalence between sheaf and group cohomology \cite{Mumford}, $H^1(M,\OOO) \cong H^1(G,N)$, where
$G$ is generated by the supertranslations $S,T$ of \eqref{Stranslation}, \eqref{Ttranslation} acting on the module of functions $N$ on $\CC^{1 \mid n}$. 
Proposition (B.1.3) of \cite{Silverman} gives an exact sequence
\begin{equation}
0 \rightarrow H^1((T),N^S) \rightarrow H^1(G,N) \rightarrow H^1((S),N)
\end{equation} 
whose final term vanishes.
Thus $H^1(G,N)$ can be computed as the cohomology of the cyclic group generated by $T$ acting on the $S$-invariant functions in $N$.
Since this cohomology is also known to be independent of $z$, we can take $S$ to be the identity and simply compute the
cohomology of $(T)$ acting on $E_n$. 

Thus, for our purposes, the group $G = \langle g \rangle$ is the infinite cyclic group and the module $N$ is the exterior algebra $E_n$.
The generator $g$ of $G$ acts on $E_n$ by the translation operator $T$:
\begin{equation}
g \cdot f := T(f) \quad \quad f \in E_n.
\end{equation}
We may identify the group of $1$-cocycles with 
$E_n$ itself. Indeed, given an element
$f \in E_n$, the corresponding 1-cocycle $c_f: G \rightarrow M$ determined by
$c_f(g) = f$.  This motivates the following definition.

\begin{defn}
\label{h1-definition}
Let $H^1(M,\OOO)$ be the quotient of $E_n$ by its linear subspace
\begin{equation}
\{ T(f) - f \,:\, f \in E_n \}
\end{equation}
of 1-coboundaries.
\end{defn}

The denominator in the quotient space of Definition~\ref{h1-definition} may be written in 
another way using the operator $\tau$.
This will allow us to deduce that, like $H^0(M,\OOO)$, the first cohomology $H^1(M,\OOO)$ is bigraded.

\begin{proposition}
\label{h1-alternative-characterization}
The set $\{ T(f) - f \,:\, f \in E_n \}$ of 1-coboundaries equals the image
of the map
\begin{equation}
\tau: E_n \rightarrow E_n.
\end{equation}
Thus, the quotient  $H^1(M,\OOO) = \bigoplus_{i,j = 0}^n H^1(M,\OOO)_{i,j}$ is the cokernel of the bihomogeneous map $\tau$.
In particular, the vector  space $H^1(M,\OOO)$ is bigraded.
\end{proposition}

\begin{proof}
Equation~\ref{exponential-expansion} shows that for any $f \in E_n$ we have
\begin{equation}
T(f) - f = \tau(f) + \frac{1}{2!} \tau^2(f) + \frac{1}{3!} \tau^3(f) + \cdots = \tau \left( f + \frac{1}{2!} \tau(f) + \frac{1}{3!} \tau^2(f) + \cdots    \right)
\end{equation}
so that every 1-coboundary is in the image of $\tau$.

The reverse containment follows from induction on $\theta$-degree. If $f \in E_n$ 
has $\theta$-degree 0, then $\tau(f) = 0 = T(0) - 0$ is a 1-coboundary. In general, we have
\begin{equation}
\tau(f) = T(f) - f - \frac{1}{2!} \tau^2(f) - \frac{1}{3!} \tau^3(f) + \cdots = 
(T(f) - f) - \tau \left(  \frac{1}{2!} \tau(f) + \frac{1}{3!} \tau^2(f) + \cdots   \right).
\end{equation}
The term $T(f) - f$ is a 1-coboundary by definition and the sum $\frac{1}{2!} \tau(f) + \frac{1}{3!} \tau^2(f) + \cdots $ has strictly lower
$\theta$-degree than $f$, so its image  
$\tau \left(  \frac{1}{2!} \tau(f) + \frac{1}{3!} \tau^2(f) + \cdots   \right)$ under $\tau$ is a 1-coboundary by induction.
\end{proof}

Since $\tau$ is $\symm_n$-equivariant, the bigraded vector space $H^1(M,\OOO) = \bigoplus_{i,j = 0}^n H^1(M,\OOO)_{i,j}$ attains the status 
of a bigraded $\symm_n$-module. 
The structure of this module may be determined using the same ideas as in the proof of 
Theorem~\ref{h0-module-structure}.

\begin{theorem}
\label{h1-module-structure}
The vector space $H^1(M,\OOO)_{i,j}$ is zero unless $i \leq j$.  When $i \leq j$, we have the equality
\begin{align}
\left[  H^1(M,\OOO)_{i,j} \right] &= 
\left[  \wedge^i \CC^n \right] \cdot \left[  \wedge^j \CC^n \right] -
\left[ \wedge^{i-1} \CC^n \right] \cdot \left[  \wedge^{j+1} \CC^n \right] \\
&= [ (E_n)_{i,j} ] - [(E_n)_{i-1,j+1}]
\end{align}
in the Grothendieck ring of $\symm_n$. Here $\symm_n$ acts on $\CC^n$ by coordinate permutation.
\end{theorem}

\begin{proof}
As in the proof of Theorem~\ref{h0-module-structure}, the map
\begin{equation}
\tau: (E_n)_{i,j} \longrightarrow (E_n)_{i+1,j-1}
\end{equation}
is injective whenever $i < j$ and surjective whenever $i \geq j$. Since $\tau$ is $\symm_n$-equivariant, the result
follows from Proposition~\ref{h1-alternative-characterization}.
\end{proof}

Let $V$ and $W$ be finite-dimensional complex vector spaces and let  $\langle -, - \rangle: V \otimes W \rightarrow \CC$ be a bilinear pairing between them.
The pairing $\langle -, - \rangle$ is 
{\em perfect} if for all nonzero vectors $v \in V$, there exists a vector $w \in W$ such that $\langle v, w \rangle \neq 0$ and 
for all nonzero vectors $w \in W$, there exists a vector $v \in V$ such that $\langle v, w \rangle \neq 0$.
In particular, if a perfect pairing exists, then $\dim V = \dim W$.

Let $X$ be an $n$-dimensional smooth complex projective variety with canonical line bundle $\KKK_X$. If $\EEE$ is an algebraic vector bundle over $X$,
{\em Serre duality} furnishes a perfect pairing
\begin{equation}
H^i(X,\EEE) \otimes H^{n-i}(X,\KKK_X \otimes \EEE^*) \longrightarrow \CC
\end{equation}
between sheaf cohomology groups of complementary degree. This is an analogue of Poincar\'e duality for sheaf cohomology.

We describe a supergeometric version of Serre duality which holds between the bigraded rings $H^0(M,\OOO)$ and $H^1(M,\OOO)$.  
Consider the `volume form' $\vol_n := \alpha_1 \cdots \alpha_n \cdot \theta_1 \cdots \theta_n$ in
$E_n$
and define a bilinear pairing $\langle - , - \rangle$ on 
$E_n$
by
\begin{equation}
\langle f, g \rangle := \text{coefficient of $\vol_n$ in $f \cdot g$}.
\end{equation}

\begin{theorem}
\label{perfect-pairing}
The pairing $\langle -, - \rangle$ induces a perfect pairing 
$H^0(M,\OOO) \otimes H^1(M,\OOO) \rightarrow \CC$.
\end{theorem}

An explicit set of coset representatives for elements in $H^1(M,\OOO)$ can be obtained, if desired,
from Theorem~\ref{perfect-pairing}.

\begin{proof}
The pairing $\langle - , - \rangle$ is easily seen to be perfect as a map
\begin{equation}
E_n \otimes  E_n \longrightarrow \CC.
\end{equation}
We claim that the operator $\tau$ is self-adjoint with respect to this pairing; that is
\begin{equation}
\label{tau-self-adjoint}
\langle \tau(f), g \rangle = \langle f, \tau(g) \rangle \quad \quad \text{for all $f, g \in E_n$.}
\end{equation}
Indeed, it suffices to check Equation~\eqref{tau-self-adjoint} in the case where $f$ and $g$ are monomials, 
and this  is a straightforward computation.

Propositions~\ref{h0-is-bigraded} and \ref{h1-alternative-characterization} together with 
Equation~\eqref{tau-self-adjoint} imply that $\langle - , - \rangle$ descends to a well-defined bilinear pairing
\begin{equation}
\langle - , - \rangle: H^0(M,\OOO) \otimes H^1(M,\OOO) \longrightarrow \CC.
\end{equation}
Indeed, if $f \in H^0(M,\OOO) = \mathrm{Ker}(\tau)$ and $g \in E_n$ then
$\langle f, \tau(g) \rangle  = \langle \tau(f), g \rangle = \langle 0, g \rangle = 0$.
Since $\dim H^0(M,\OOO)_{i,j} = \dim H^1(M,\OOO)_{n-i,n-j}$, the pairing $\langle - , - \rangle$ remains perfect when regarded as a  map 
$H^0(M,\OOO) \otimes H^1(M,\OOO) \rightarrow \CC.$
\end{proof}

\begin{remark}
The proofs in this section made heavy use of the operator $\tau = \sum_{i = 1}^n \alpha_i \cdot (\partial/\partial \theta_i)$ on the exterior algebra $E_n$.
This operator comes from an action of the Lie algebra $\sell_2(\CC)$ on $E_n$.
Let  $\sigma: E_n \rightarrow E_n$ be the linear operator
\begin{equation}
\sigma(f) := \sum_{i = 1}^n \theta_i \cdot (\partial/\partial \alpha_i) f
\end{equation}
and let $\eta: E_n \rightarrow E_n$ be the linear operator which acts on the subspace $(E_n)_{i,j}$ by the scalar $i - j$.
It can be shown that
\begin{equation}
\label{sl2-relations}
[\tau, \sigma] = \eta \quad \quad 
[\eta, \tau] = 2 \cdot \tau \quad \quad
[\eta, \sigma] = -2 \cdot \sigma
\end{equation}
as operators on $E_n$. Indeed, the relation $[\tau, \sigma] = \eta$ is easily checked on monomials while the relations
$[\eta, \tau] = 2 \cdot \tau$ and 
$[\eta, \sigma] = -2 \cdot \sigma$ follow from bidegree considerations.
Since \eqref{sl2-relations} are the defining relations of  $\sell_2(\CC)$, we see that the action of $\tau, \sigma,$ and $\eta$ endows $E_n$ 
with the structure of an $\sell_2(\CC)$-module.

Similarly, it can be seen that the Lefschetz element $\ell = \sum_i \alpha_i \theta_i$ of bidegree $(1,1)$ appearing in Theorem~\ref{lefschetz-theorem}
fits into an $\sell_2(\CC)$-action with the operator $\sum_i (\partial/\partial \theta_i) (\partial/\partial \alpha_i)$ of bidegree $(-1,-1)$.
Here the Cartan generator of $\sell_2(\CC)$ acts on $(E_n)_{i,j}$ by the scalar $i+j-n$.
\end{remark}

\section{Generators and relations for $H^0(M,\OOO)$}
\label{Ring}

In the last section, we studied the $T$-invariant subalgebra
 $H^0(M,\OOO) \subseteq E_n$ as a doubly-graded $\symm_n$-module. In this section, we focus on its ring structure and give  a simple set of
 generators and combinatorial relations for $H^0(M,\OOO)$.
 
 It can be easily seen that the elements $\alpha_i$, $\alpha_i\theta_i$, and $\alpha_i\theta_j + \alpha_j\theta_i$ are all translation invariant (we will informally refer to these as ``basic invariants"), as are any products of these elements. But there are many dependence relations among products of these elements, and it is not immediately obvious how to pick out a basis from said products. This section will explain how to do so thereby giving a combinatorial basis for the ring of translation invariants.

We will use labelled matchings to index our basis. A {\em matching of size $n$} consists of a collection of pairwise disjoint size-two subsets of $\{1, \dots, n\}$. For our purposes we do not require that this collection partition $\{1, \dots, n\}$; there may be unmatched elements. 
To visualize matchings, we will depict them as $n$ vertices labelled $1$ through $n$, placed on a line, with an arc connecting matched vertices, e.g. when $n=8$,

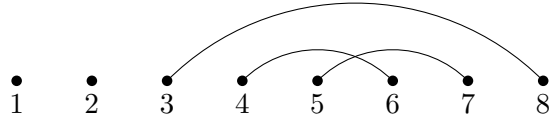
\begin{figure}[h]
\begin{center}
\begin{tikzpicture}
\draw (0,0) node{1};
\fill[black] (0,.3) circle (.07 cm);
\draw (1,0) node{2};
\fill[black] (1,.3) circle (.07 cm);
\draw (2,0) node{3};
\fill[black] (2,.3) circle (.07 cm);
\draw (3,0) node{4};
\fill[black] (3,.3) circle (.07 cm);
\draw (4,0) node{5};
\fill[black] (4,.3) circle (.07 cm);
\draw (5,0) node{6};
\fill[black] (5,.3) circle (.07 cm);
\draw (6,0) node{7};
\fill[black] (6,.3) circle (.07 cm);
\draw (7,0) node{8};
\fill[black] (7,.3) circle (.07 cm);
\draw (3,.3) to[out = 45, in = 135] (5,.3);
\draw (4,.3) to[out = 45, in = 135] (6,.3);
\draw (2,.3) to[out = 45, in = 135] (7,.3);
\end{tikzpicture}
\end{center}
\end{figure}

Let $\Phi(n)$ consist of all matchings of size $n$ in which each 
unmatched element is either labelled $\alpha$, $\alpha\theta$ or unlabelled. 
To each matching $m \in \Phi(n)$ we  associate  $F_m \in H^0(M,\OOO)$ as follows.
For each unmatched vertex $i$ with an $\alpha$ label, take the product of the $\alpha_i$'s in increasing numerical order. 
Then, multiply by $\alpha_i\theta_i$ for each vertex $i$ labelled $\alpha\theta$. Finally, multiply by $\alpha_i\theta_j + \alpha_j\theta_i$ for each matched pair $\{i,j\}$. 
Note that these terms are degree $2$ and therefore commute, so there is no ambiguity in the order in which they are multiplied.
For example, if $m \in \Phi(8)$ is the labelled matching shown below 
\begin{figure}[h]
\begin{center}
\begin{tikzpicture}
\draw (0,0) node{1};
\fill[black] (0,.3) circle (.07 cm);
\draw (1,0) node{2};
\fill[black] (1,.3) circle (.07 cm);
\draw (2,0) node{3};
\fill[black] (2,.3) circle (.07 cm);
\draw (3,0) node{4};
\fill[black] (3,.3) circle (.07 cm);
\draw (4,0) node{5};
\fill[black] (4,.3) circle (.07 cm);
\draw (5,0) node{6};
\fill[black] (5,.3) circle (.07 cm);
\draw (6,0) node{7};
\fill[black] (6,.3) circle (.07 cm);
\draw (7,0) node{8};
\fill[black] (7,.3) circle (.07 cm);
\draw (3,.3) to[out = 45, in = 135] (5,.3);
\draw (4,.3) to[out = 45, in = 135] (6,.3);
\draw (0,.6) node{$\alpha$};
\draw (1,.6) node{$\alpha\theta$};
\end{tikzpicture}
\end{center}
\end{figure}
\noindent
then $F_m \in H^0(M,\OOO)$ is given by
$\alpha_1 \cdot \alpha_2\theta_2 \cdot (\alpha_4\theta_6 + \alpha_6\theta_4)(\alpha_5\theta_7 + \alpha_7\theta_5).$

The benefit of visualizing products of basic invariants as matchings lies in  
an easier description of the linear dependence relations among these products,
as in the following lemmas. The first is the Ptolemy relation (up to sign).

\begin{lemma}
\label{uncross}
Let $m \in \Phi(n)$ and suppose that $m$ contains arcs $(i,k)$ and $(j,l)$ with $i<j<k<l$. Let $m_0$ and $m_1$ be the 
matchings obtained by replacing arcs $(i,k)$ and $(j,l)$ with arcs $(i,j)$ and $(k,l)$ or with arcs $(i,l)$ and $(j,k)$ respectively. Then
\[
F_m+F_{m_0} + F_{m_1} = 0.
\]
Pictorially:
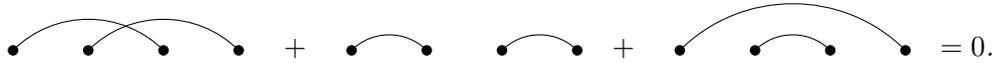
\begin{center}
\begin{tikzpicture}
\fill[black] (3,.3) circle (.07 cm);
\fill[black] (4,.3) circle (.07 cm);
\fill[black] (5,.3) circle (.07 cm);
\fill[black] (6,.3) circle (.07 cm);
\draw (3,.3) to[out = 45, in = 135] (5,.3);
\draw (4,.3) to[out = 45, in = 135] (6,.3);
\end{tikzpicture}
\quad $\mathrm{+}$ \quad
\begin{tikzpicture}
\fill[black] (3,.3) circle (.07 cm);
\fill[black] (4,.3) circle (.07 cm);
\fill[black] (5,.3) circle (.07 cm);
\fill[black] (6,.3) circle (.07 cm);
\draw (3,.3) to[out = 45, in = 135] (4,.3);
\draw (5,.3) to[out = 45, in = 135] (6,.3);
\end{tikzpicture}\quad $\mathrm{+}$ \quad
\begin{tikzpicture}
\fill[black] (3,.3) circle (.07 cm);
\fill[black] (4,.3) circle (.07 cm);
\fill[black] (5,.3) circle (.07 cm);
\fill[black] (6,.3) circle (.07 cm);
\draw (3,.3) to[out = 45, in = 135] (6,.3);
\draw (4,.3) to[out = 45, in = 135] (5,.3);
\end{tikzpicture}\quad $= 0$.
\end{center}
\end{lemma}

The second describes a relationship between $\alpha$-labelled vertices and arcs.

\begin{lemma}
\label{movealpha}
Let $m \in \Phi(n)$ and suppose that $d$ contains the arc $(i,k)$ and vertex $j$ is labelled by $\alpha$ with $i<j<k$. 
Let $m_0$ and $m_1$ be the diagrams obtained by replacing arc $(i,k)$ with arc $(i,j)$ and vertex $k$ labelled by $\alpha$ 
or with arc $(j,k)$ and vertex $i$ labelled by $\alpha$ respectively. Then
\[
F_m+F_{m_0} + F_{m_1} = 0.
\]
Pictorially:
\begin{center}
\begin{tikzpicture}
\fill[black] (3,.3) circle (.07 cm);
\fill[black] (4,.3) circle (.07 cm);
\fill[black] (5,.3) circle (.07 cm);
\draw (3,.3) to[out = 45, in = 135] (5,.3);
\draw (4,.6) node{$\alpha$};
\end{tikzpicture}
\quad $\mathrm{+}$ \quad
\begin{tikzpicture}
\fill[black] (3,.3) circle (.07 cm);
\fill[black] (4,.3) circle (.07 cm);
\fill[black] (5,.3) circle (.07 cm);
\draw (3,.3) to[out = 45, in = 135] (4,.3);
\draw (5,.6) node{$\alpha$};
\end{tikzpicture}\quad $\mathrm{+}$ \quad
\begin{tikzpicture}
\fill[black] (3,.3) circle (.07 cm);
\fill[black] (4,.3) circle (.07 cm);
\fill[black] (5,.3) circle (.07 cm);
\draw (4,.3) to[out = 45, in = 135] (5,.3);
\draw (3,.6) node{$\alpha$};
\end{tikzpicture}\quad $= 0$.
\end{center}
\end{lemma}

Lemmas~\ref{uncross} and \ref{movealpha} are both proven by direct computation; we leave this to the reader.
These lemmas allow us to reduce the linearly dependent set of translation invariants
\begin{equation}
\{F_m \,:\, m \in \Phi(n)\}
\end{equation}
to
 a linearly independent subset.
 Lemma~\ref{uncross} allows translation invariants corresponding to matchings to be written as a linear combination of translation invariants corresponding to matchings with fewer 
 crossings or fewer nestings. 
 Lemma~\ref{movealpha} allows translation invariants corresponding to matchings to be written as a linear combination of 
 translation invariants corresponding to matchings with fewer $\alpha$ labels under arcs. The next two lemmas make this idea explicit.

Let $m \in \Phi(m)$ and let $C(m)$ denote the number of crossings of $m$,
 i.e. the number of quadruples $i<j<k<\ell$ such that $\{i,k\}$ and $\{j,\ell\}$ are matched in $m$. 
 Similarly, let $A(m)$ denote the number of times an $\alpha$ label appears under an arc of $m$,
 i.e. the number of triples $i<j<k$ such that $j$ is unmatched and has an $\alpha$ label and $\{i,k\}$ is matched in $m$. 

\begin{lemma}
\label{crossing-nesting-reduction}
If $m \in \Phi(n)$, then $m$ satisfies one of the following
\begin{itemize}
\item $C(m) = A(m) = 0$
\item $C(m)=0$ and there exists a collection of matchings $m_i\in \Phi(n)$ with $A(m_i) < A(m)$ such that $F_m = \sum_{i} c_i F_{m_i}$ for some constants $c_i$.
\item There exists a collection of matchings $m_i\in \Phi(n)$ with $C(m_i) < C(m)$ with $F_m = \sum_{i} c_i F_{m_i}$ for some constants $c_i$.
\end{itemize}
\end{lemma}
\begin{proof}
Firstly, suppose $C(m) \neq 0$. So there exists at least on crossing, and we can apply Lemma~\ref{uncross} to satisfy the third condition. 

Otherwise, suppose $C(m) = 0$ and $A(m)\neq 0$. Then there exists at least one $\alpha$ labelled vertex lying under an arc. Applying Lemma~\ref{movealpha} to such a vertex and the shortest arc it lies under  will allow us to express $F_m$ as a linear combination of $F_{m_1}$ and $F_{m_2}$ for two matchings $m_1$ and $m_2$ as defined in Lemma~\ref{movealpha}. Note that since $m$ has no crossings, and we chose the shortest arc, $m_1$ and $m_2$ will also have no crossings, and $A(m_1) < A(m)$ and $A(m_2) <A(m)$, so the second condition is satisfied. 

Otherwise, the first condition is satisfied.
\end{proof}

Let $NC(n) \subseteq \Phi(n)$ denote the set of noncrossing matchings in which no arc nests an $\alpha$ label. That is, we set
\begin{equation}
NC(n) := 
\{m \in \Phi(n) \,:\, C(m) = A(m) = 0\}. 
\end{equation}
Lemma~\ref{crossing-nesting-reduction} has the following corollary.

\begin{corollary}
\label{spanning-corollary}
For any $m_0 \in \Phi(n)$, the element $F_{m_0} \in H^0(M,\OOO)$ lies in the span of 
\begin{equation}
\{ F_m \,:\, m \in NC(n) \}.
\end{equation}
\end{corollary}

It will turn out that the set in Corollary~\ref{spanning-corollary} is a basis for $H^0(M,\OOO)$.
We will now show its linearly independence using a nested induction argument.
 Let $NC(n,k)$ denote the set of all $m \in NC(n)$ such that $F_m$ is of total degree $k$.

\begin{proposition}
\label{prop:independent}
The set $\{F_m \,:\, m \in NC(n,k)\}$ is linearly independent in $H^0(M,\OOO)$.
\end{proposition}
\begin{proof}
For any labelled matching $m$, let $s(m)$ denote the index of the smallest vertex not labelled by $\alpha$ or $\alpha\theta$ in $m$, or $n+1$ if no such vertex exists.

This proof proceeds via two nested inductions, the first downwards on $k$ and the second on $s(m)$.

When  $k=2n$, there is a unique matching in $NC(n,2n)$. In this matching, everything is labelled by $\alpha\theta$. 

Now assume as an inductive hypothesis that for some $k<2n$ the set $\{F_m \,:\, m \in NC(n,k+1)\}$ is linearly independent, and suppose
\begin{equation}
\sum_{m\in NC(n,k)} c_m F_m = 0.
\end{equation}

For clarity, we begin with the simpler case where $s(m) = 1$; this is the base case of our induction on $s(m)$.
To show that for every matching where $s(m) = 1$ we have $c_m =0$, multiply the given linear dependence by $\alpha_1$. We have
\[
\sum_{m\in NC(n,k)} c_m \alpha_1 F_m = 0.
\]
For any matching $m$ with $s(m) >1$, vertex 1 is either labelled $\alpha$ or $\alpha\theta$. In either case, $\alpha_1F_m = 0$. Otherwise, if $s(m) = 1$, then $\alpha_1F_m = \pm F_{m'}$ where $m'$ is the matching obtained from $m$ by either
\begin{enumerate}
\item Adding an $\alpha$ label to vertex $1$ if vertex $1$ was not part of any arc.
\item Removing arc $(1,j)$, then adding an $\alpha\theta$ label to vertex $1$ and adding an $\alpha$ label to vertex $j$, if vertex $1$ was part of arc $(1,j)$. 
\end{enumerate}
This is because $\alpha_1(\alpha_1\theta_j + \alpha_j\theta_1) = \alpha_1 \alpha_j\theta_1$. An example of case $2$ is shown in the 
figure below.
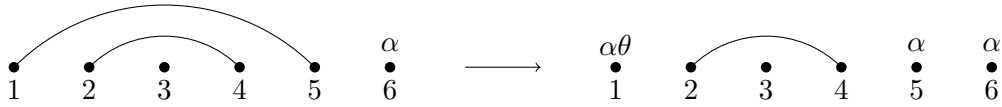
\begin{center}
\begin{tikzpicture}
\draw (0,0) node{1};
\fill[black] (0,.3) circle (.07 cm);
\draw (1,0) node{2};
\fill[black] (1,.3) circle (.07 cm);
\draw (2,0) node{3};
\fill[black] (2,.3) circle (.07 cm);
\draw (3,0) node{4};
\fill[black] (3,.3) circle (.07 cm);
\draw (4,0) node{5};
\fill[black] (4,.3) circle (.07 cm);
\draw (5,0) node{6};
\fill[black] (5,.3) circle (.07 cm);
\draw (0,.3) to[out = 45, in = 135] (4,.3);
\draw (1,.3) to[out = 45, in = 135] (3,.3);
\draw (5,.6) node{$\alpha$};

\draw[->] (6,.3)--(7,.3);
\draw (8,0) node{1};
\fill[black] (8,.3) circle (.07 cm);
\draw (9,0) node{2};
\fill[black] (9,.3) circle (.07 cm);
\draw (10,0) node{3};
\fill[black] (10,.3) circle (.07 cm);
\draw (11,0) node{4};
\fill[black] (11,.3) circle (.07 cm);
\draw (12,0) node{5};
\fill[black] (12,.3) circle (.07 cm);
\draw (13,0) node{6};
\fill[black] (13,.3) circle (.07 cm);
\draw (9,.3) to[out = 45, in = 135] (11,.3);
\draw (8,.6) node{$\alpha\theta$};
\draw (12,.6) node{$\alpha$};
\draw (13,.6) node{$\alpha$};
\end{tikzpicture}
\end{center}
Note that in either case, $m'$ is in $NC(n,k+1)$, since neither vertex $1$ nor vertex $j$ can lie under any arc in $m'$. In case 2, if vertex $j$ did, that arc would have had to cross arc $(1,j)$ in $m$. By our first inductive hypothesis all such $m'$ are linearly independent, so $c_m = 0$ for any $m$ with $s(m) = 1$.



To complete the inductive step, assume as an inductive hypothesis that $c_m = 0$ for any $m$ with $s(m) < \ell$. Multiply the linear dependence by $\alpha_\ell$. We have
\[
\sum_{m\in NC(n,k)} c_m \alpha_\ell F_m = 0.
\]
For any $m$ with $s(m) > \ell$, $\alpha_\ell F_m = 0$. For any $m$ with $s(m) < \ell$, our inductive hypothesis assumed $c_m = 0$. The analogous argument to the $s(m)=1$ case therefore shows that $c_m =0$ for all matchings $m$ with $s(m) = \ell$.
Therefore by strong induction $c_m = 0$ for all $m$, and the set $\{F_m \,:\, m \in NC(n,k)\}$ is linearly independent.
\end{proof}

We now know that the set 
$\{F_m \,:\, m \in NC(n)\}$
is linearly independent and spans all products of basic invariants. To show that it is a basis for the entire translation invariant subring, we employ a dimension count. To count the size of our proposed basis, we will give a bijection.

\begin{proposition}
\label{prop:bijection}
$NC(n,k)$ is in bijection with ordered pairs $(A,B)$ of subsets of $\{1,\dots,n\}$, where $A$ has size $\lfloor k/2 \rfloor$ and $B$ has size $\lceil k/2 \rceil$. 
\end{proposition}

\begin{proof}
Given any pair of subsets $A,B \subseteq [n]$ with $|A| = \lfloor k/2 \rfloor$ and $|B| = \lceil k/2 \rceil$, there is a unique $m \in NC(n,k)$ satisfying the following conditions.
\begin{enumerate}
\item The set of unmatched elements labelled $\alpha\theta$ is $A \cap B$.
\item The set of unmatched unlabelled elements is $\{1,\dots,n\} - (A \cup B)$.
\item The smaller element of any matched pair (i.e. the left endpoint of any arc) is in $A$.
\item The larger element of any matched pair (i.e. the right endpoint of an arc) is in $B$.
\item Every unmatched element in $B$ is to the left of every unmatched element of $A$.
\end{enumerate}

These conditions force every element $a$ of $A\setminus B$ to be matched with the smallest element $b$ of $B\setminus A$ for which the number of elements of $A$ between $a$ and $b$ matches the number of elements of $B$ between $a$ and $b$, if such an element exists, and otherwise be unmatched and labelled $\alpha$. To see this, consider an example where $n=8$, $A = \{1,2,4,5\}$ and $B = \{3,4,6,7,8\}$. Everything in $A\setminus B$ is either a left endpoint or labelled $\alpha$, and similarly for $B \setminus A$. We have the following picture:

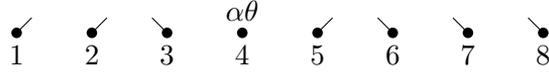
\begin{figure}[h]
\begin{center}
\begin{tikzpicture}
\draw (0,0) node{1};
\fill[black] (0,.3) circle (.07 cm);
\draw (1,0) node{2};
\fill[black] (1,.3) circle (.07 cm);
\draw (2,0) node{3};
\fill[black] (2,.3) circle (.07 cm);
\draw (3,0) node{4};
\fill[black] (3,.3) circle (.07 cm);
\draw (4,0) node{5};
\fill[black] (4,.3) circle (.07 cm);
\draw (5,0) node{6};
\fill[black] (5,.3) circle (.07 cm);
\draw (6,0) node{7};
\fill[black] (6,.3) circle (.07 cm);
\draw (7,0) node{8};
\fill[black] (7,.3) circle (.07 cm);
\draw (0,.3)--(.2,.5);
\draw (1,.3)--(1.2,.5);
\draw (4,.3)--(4.2,.5);
\draw (2,.3)--(1.8,.5);
\draw (5,.3)--(4.8,.5);
\draw (6,.3)--(5.8,.5);
\draw (7,.3)--(6.8,.5);
\draw (3,.6) node {$\alpha\theta$};
\end{tikzpicture}
\end{center}
\end{figure}

We must connect left endpoints with right endpoints so that arcs do not cross, unmatched endpoints do not lie under arcs, and unmatched left endpoints do not lie to the left of any unmatched right endpoint. 

Consider which endpoint vertex 1 can be connected to. If it were connected to 3 or 6, then at least one of 2 or 5 would be left unmatched and under an arc. So it cannot connect to 3 or 6. Continuing this reasoning, we see it can only be connected to 7. Similarly, there is a unique option for every other endpoint, and the only matching is shown below.


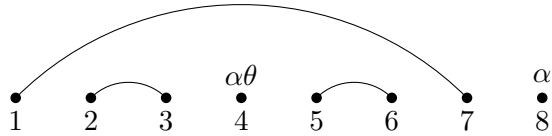
\begin{figure}[h]
\begin{center}
\begin{tikzpicture}
\draw (0,0) node{1};
\fill[black] (0,.3) circle (.07 cm);
\draw (1,0) node{2};
\fill[black] (1,.3) circle (.07 cm);
\draw (2,0) node{3};
\fill[black] (2,.3) circle (.07 cm);
\draw (3,0) node{4};
\fill[black] (3,.3) circle (.07 cm);
\draw (4,0) node{5};
\fill[black] (4,.3) circle (.07 cm);
\draw (5,0) node{6};
\fill[black] (5,.3) circle (.07 cm);
\draw (6,0) node{7};
\fill[black] (6,.3) circle (.07 cm);
\draw (7,0) node{8};
\fill[black] (7,.3) circle (.07 cm);
\draw (3,.6) node {$\alpha\theta$};

\draw (1,.3) to[out = 45, in = 135] (2,.3);

\draw (4,.3) to[out = 45, in = 135] (5,.3);
\draw (0,.3) to[out = 45, in = 135] (6,.3);
\draw (7,.6) node{$\alpha$};
\end{tikzpicture}
\end{center}
\end{figure}

To invert this map, put every left endpoint in $A$, every right endpoint in $B$, and everything labelled $\alpha\theta$ in both. For vertices labelled $\alpha$, put the leftmost correct amount (so that $|B| = \lceil \frac{k}{2} \rceil$) in $B$ and the rightmost correct amount in $A$. This gives a bijection from $NC(n,k)$ to pairs of subsets of
 $\{1,\dots,n\}$ of the appropriate size. 
\end{proof}

Proposition~\ref{prop:bijection} shows that 
\begin{equation}
| NC(n) | = \sum_k |NC(n,k)| = \sum_k {n \choose \lfloor k/2 \rfloor} {n \choose \lceil k/2 \rceil} = {2n+1 \choose n}.
\end{equation}
On the other hand,
Theorem~\ref{h0-module-structure} implies that 
\begin{equation}
\dim H^0(M,\OOO) = \sum_{1 \leq i \leq j \leq n} \left( {n \choose i} {n \choose j} - {n \choose i-1} {n \choose j+1} \right) = {2n+1 \choose n}.
\end{equation}
Thanks to
Proposition~\ref{prop:independent}, we have the following corollary.

\begin{corollary}
\label{cor:basis}
The set $\{ F_m \,:\, m \in NC(n) \}$ is a basis of $H^0(M,\OOO)$.
\end{corollary}

We can also give a simple combinatorial presentation for $H^0(M,\OOO)$ as a ring with generators and combinatorial relations.
Lemmas~\ref{uncross} and \ref{movealpha} give combinatorial relations among the basic invariants in $H^0(M,\OOO)$.  
The remaining relations among these invariants are straightforward.

\begin{theorem}
\label{h0-presentation}
The translation invariant ring $H^0(M,\OOO)$ is generated as a $\CC$-algebra by the basic invariants 
$\alpha_i, \alpha_i \theta_i$, and $\alpha_i \theta_j + \alpha_j \theta_i$
subject only to the combinatorial relations in Lemma~\ref{uncross} and Lemma~\ref{movealpha} together with 
\begin{equation}
\begin{cases}
\alpha_i^2 = 0 \\
(\alpha_i \theta_i)^2 = 0 \\
(\alpha_i \theta_i) (\alpha_i \theta_j + \alpha_j \theta_i) = 0 \\
\end{cases}  \quad \quad
\begin{cases}
(\alpha_i \theta_j + \alpha_j \theta_i)^2 = -2 (\alpha_i \theta_i) (\alpha_j \theta_j) \\
\alpha_i (\alpha_i \theta_j + \alpha_j \theta_i) = - \alpha_j \cdot (\alpha_i \theta_i) \\
(\alpha_i \theta_i) \cdot (\alpha_i \theta_j + \alpha_j \theta_i) = 0
\end{cases}
\end{equation}
\end{theorem}

\begin{proof}
Corollary~\ref{cor:basis} shows that the basic invariants generate $H^0(M,\OOO)$ as a $\CC$-algebra, so we only need to show that any relation among products
of basic invariants can be deduced from the given relations.
The relations displayed in braces in the statement of the theorem may be used to show that any nonzero product of basic invariants is, up to a scalar, 
a product $F_{m_0}$ for some $m_0 \in \Phi(n)$.
Lemma~\ref{crossing-nesting-reduction} shows that, using only the relations of Lemmas~\ref{uncross} and \ref{movealpha},
the product $F_{m_0}$ can be written as a linear combination of $F_{m}$'s for various $m \in NC(n)$.
Corollary~\ref{cor:basis} says that $\{ F_m \:\, m \in NC(n) \}$ is a basis of $H^0(M,\OOO)$, so every relation among the basic invariants is given in the statement of the 
theorem.
\end{proof}

\section{Conclusion}
\label{Conclusion}

Throughout this paper, we assumed that the coefficient variables $\alpha_1, \alpha_2, \dots $ defining the translation $T$ were independent fermionic parameters.
However, in supergeometry one often considers more general fermionic translations
$T: E_n \rightarrow E_n$ of the form
$(\theta_1, \dots, \theta_n) \mapsto (\theta_1 + \beta_1, \dots, \theta_n + \beta_n)$ where the $\beta_i$ are elements of odd degree in $\alpha_1, \dots, \alpha_n$, i.e.
\begin{equation}
\beta_1, \dots, \beta_n \in \bigoplus_{i \geq 0} \wedge^{2i+1} \{\alpha_1, \dots, \alpha_n \}.
\end{equation}
In particular,
the elements $\beta_1, \dots, \beta_n$ may satisfy nontrivial relations.

Let $R \subseteq E_n$ be the subalgebra of $E_n$ which is invariant under the action of $T$. Since $T(\alpha_i) = \alpha_i$ for all $i$, the algebra $R$ has the structure
of a module over the exterior algebra $\wedge \{ \alpha_1, \dots, \alpha_n \}$.

\begin{problem}
\label{beta-problem}
Describe the structure of $R$ as a module over the free exterior algebra  $\wedge \{ \alpha_1, \dots, \alpha_n \}$.
\end{problem}

 Theorems~\ref{h0-module-structure} and \ref{h0-presentation} solve Problem~\ref{beta-problem} when the $\beta_i$ are independent.
 The general case is more difficult because relations among the $\beta_i$ can generate additional invariants in $R$ in ways that are difficult to predict.

\section{Acknowledgements}
\label{Acknowledgements}

B. Rhoades was partially supported by NSF Grant DMS-1953781.
The authors are grateful to Vera Serganova for suggesting that the action of $\ell$ on $E_n$ should come from a larger action of $\sell_2(\CC)$.

\end{document}